\documentclass[11pt,A4paper]{article}

\usepackage[left=32mm,right=32mm,top=30mm,bottom=32mm]{geometry}
\usepackage{labelfig}
\usepackage{epsfig}
\usepackage{epstopdf}

\usepackage{xfrac}

\usepackage[percent]{overpic}

\usepackage{color}
\usepackage{amsthm,amsmath,amssymb}
\usepackage{booktabs}
\usepackage{mathpazo}
\usepackage{microtype}
\usepackage{overpic}
\usepackage{bm}
\usepackage{sectsty}
\usepackage[
	pdftitle={PDFTitle},
	pdfauthor={Hugo Parlier},
	ocgcolorlinks,
	linkcolor=linkred,
	citecolor=linkred,
	urlcolor=linkblue]
{hyperref}

\definecolor{linkred}{rgb}{0.48,0.1,0.05}
\definecolor{linkblue}{RGB}{16, 78, 139}

\usepackage[hang,flushmargin]{footmisc}
\usepackage{enumitem}
\usepackage{titlesec}
	\titlespacing{\section}{0pt}{12pt}{0pt}
	\titlespacing{\subsection}{0pt}{6pt}{0pt}
	
\titlelabel{\thetitle.\quad}

\makeatletter 

\long\def\@footnotetext#1{%
\H@@footnotetext{%
\ifHy@nesting 
\hyper@@anchor{\@currentHref}{#1}%
\else 
\Hy@raisedlink{\hyper@@anchor{\@currentHref}{\relax}}#1%
\fi 
}}

\def\@footnotemark{%
\leavevmode 
\ifhmode\edef\@x@sf{\the\spacefactor}\nobreak\fi 
\H@refstepcounter{Hfootnote}%
\hyper@makecurrent{Hfootnote}%
\hyper@linkstart{link}{\@currentHref}%
\@makefnmark 
\hyper@linkend 
\ifhmode\spacefactor\@x@sf\fi 
\relax 
}%

\ifFN@multiplefootnote%
\renewcommand*\@footnotemark{%
\leavevmode 
\ifhmode 
\edef\@x@sf{\the\spacefactor}%
\FN@mf@check 
\nobreak 
\fi 
\H@refstepcounter{Hfootnote}%
\hyper@makecurrent{Hfootnote}%
\hyper@linkstart{link}{\@currentHref}%
\@makefnmark 
\hyper@linkend 
\ifFN@pp@towrite 
\FN@pp@writetemp 
\FN@pp@towritefalse 
\fi 
\FN@mf@prepare 
\ifhmode\spacefactor\@x@sf\fi 
\relax%
}%
\fi 

\makeatother 

\theoremstyle{plain}
\newtheorem{theorem}{Theorem}[section]

\newtheorem{lemma}[theorem]{Lemma}
\newtheorem{corollary}[theorem]{Corollary}

\theoremstyle{definition}

\newcommand{\R}{{\mathbb R}}

\newcommand{\M}{{\mathcal M}}
\newcommand{\B}{{\mathcal B}}

\newcommand{\area}{{\rm area}}
\newcommand{\arcsinh}{{\,\rm arcsinh}}
\newcommand{\arccosh}{{\,\rm arccosh}}

\sectionfont{\large \sc}
\subsectionfont{\normalsize}

\long\def\symbolfootnote[#1]#2{\begingroup%
\def\thefootnote{\fnsymbol{footnote}}\footnote[#1]{#2}\endgroup}

\def\blfootnote{\xdef\@thefnmark{}\@footnotetext}

\begin{document}

{\Large \bfseries \sc A short note on short pants}

{\bfseries Hugo Parlier\symbolfootnote[2]{\normalsize Research supported by Swiss National Science Foundation grant number PP00P2\textunderscore 128557\\ 
{\em Address:} Department of Mathematics, University of Fribourg, Switzerland \\
{\em Email:} \href{mailto:hugo.parlier@unifr.ch}{hugo.parlier@unifr.ch}
}}

{\em Abstract.} It is a theorem of Bers that any closed hyperbolic surface admits a pants decomposition consisting of curves of bounded length where the bound only depends on the topology of the surface. The question of the quantification of the optimal constants has been well studied and the best upper bounds to date are linear in genus, a theorem of Buser and Sepp\"al\"a. The goal of this note is to give a short proof of a linear upper bound which slightly improve the best known bound.

\vspace{1cm}

\section{Introduction} \label{introduction}
A pants decomposition of a hyperbolic surface is a maximal collection of disjoint simple closed geodesics, which as its name indicates, decomposes the surface into three holed spheres or pairs of pants. In the case of closed surfaces of genus $g\geq 2$, a pants decomposition contains $3g-3$ curves which decompose the surface into $2g-2$ pairs of pants. Any surface admits an infinite number of pants decompositions and even up to homeomorphism the number of different types of pants decomposition grows quickly (roughly like $g!$). Bers proved that there exists a constant $\B_g$ which only depends on the genus $g$ such that any closed hyperbolic surface of genus $g$ has a pants decomposition with all curves of length less than $\B_g$.

The first notable step in the direction of quantifying $\B_g$ was obtained by Buser \cite{BuserHab} where an upper bounds of order $g \log g$ and lower bounds of order $\sqrt{g}$ were established. The first upper bounds linear in $g$ were obtained by Buser and Sepp\"al\"a \cite{BuserSeppala92} and Buser extended these bounds to the case of variable curvature \cite{BuserBook}. The best bounds known to date \cite[Th. 5.1.4]{BuserBook} are $6\sqrt{3\pi} (g-1)$ so the best known linear factor is $ \approx 18.4$.

It should also be noted that the direct method of computing the optimal constant in each genus seems out of reach as the only known constant is $\B_2$, a result of Gendulphe \cite{GendulpheBers}.

The goal of this note is to offer a short proof of a linear upper bounds which provide a slight improvement on previously known bounds.

\begin{theorem}\label{thm:main}
Every closed hyperbolic surface of genus $g\geq 2$ has a pants decomposition with all curves of length at most
$$
4\pi (g-1) + 4 R_g
$$
where $R_g$ is at a logarithmic function in $g$ which can be taken to be
$$
R_g= \arccosh\frac{1}{\sqrt{2} \sin\frac{\pi}{12g-6}}< \log(4g-2) + \arcsinh \,1.
$$
\end{theorem}
The theorem provides an improvement on the factor in front of the genus from $\approx 18.4$ to $\approx 12.6$. The true growth rate of $\B_g$ remains unknown. It follows from the bounds in the closed case that surfaces with $n$ cusps and genus $g$ also have short pants decompositions where the bounds depend on $n$ and $g$ this time. For fixed genus and growing number of cusps, the growth rate of the optimal constants is known to grow like $\sqrt{n}$ (see \cite{Balacheff-Parlier-Sabourau, Balacheff-Parlier}) which seems to indicate that the growth rate for the closed surfaces might be of order $\sqrt{g}$. However, if one considers sums of lengths of curves in a pants decomposition instead of the maximum length then the case of cusps is very different from the genus case (compare \cite{Balacheff-Parlier-Sabourau, Guth-Parlier-Young}). 

{\bf Acknowledgements.} 
Some of the ideas in this paper came while I was preparing a master class at the Schr\"odinger Institute in Vienna during the special program on ``Teichm\"uller Theory", Winter 2013, and I thank the organizers for inviting me. I'd also like to thank Robert Young for his helpful comments.

\section{Preliminaries} \label{preliminaries}

To a curve $\gamma$ or a free homotopy class $[\gamma]$ of curve on a topological surface $\Sigma$ we associate a length function $\ell_{S}(\gamma)$ which associates to a hyperbolic structure $S$ on $\Sigma$ the length of the unique closed geodesic in $[\gamma]$. A first tool that we shall use is the following lemma which in particular will allow us to restrict the proof of the main theorem to the case of surfaces with systole of length at least $2 \arcsinh \, 1$. 

\begin{lemma}[Length expansion lemma]\label{lem:LE} Let $\Sigma$ be a topological surface with $n > 0$ boundary curves $\gamma_1,\hdots,\gamma_n$. For any hyperbolic surface $S \cong \Sigma$ and any
$(\varepsilon_1,\hdots,\varepsilon_n) \in (\R^+)^n\setminus \{0\}$ there exists a hyperbolic surface $S'\cong \Sigma$ with
$$\ell_{S'}(\gamma_1)=\ell_S(\gamma_1)+\varepsilon_1,\hdots,\ell_{S'}(\gamma_n)=\ell_S(\gamma_n)+\varepsilon_n$$
and such that any non-trivial simple closed curve $\gamma\subset \Sigma$ satisfies
$$
\ell_{S'}(\gamma)>\ell_{S}(\gamma).
$$
\end{lemma}

This result seems to have been known for a long time, as it is claimed in \cite{ThurstonSpine} (also see \cite{ParlierLengths} for a direct proof and \cite{PapaTheretShort} for a stronger version).

The following result, due to Bavard \cite{BavardDisque}, is sharp.

\begin{lemma}[Marked systoles]\label{lem:MS} For any $x\in S$, $S$ a closed hyperbolic surface of genus $g$, there exists a geodesic loop $\delta_x$ based in $x$ such that 
$$
\ell(\delta_x) \leq 2 \arccosh \frac{1}{2 \sin \frac{\pi}{12g-6}}.
$$
\end{lemma}

What Bavard actually proves is that the above value is a sharp bound on the diameter of the largest embedded open disk of the surface. A weaker version of this lemma can be obtained by comparing the area of an embedded disk to the area of the surface. The area of $D$ an embedded disk of radius $r$ on a hyperbolic surface is the same as the area of such a disk in the hyperbolic plane so
$$
\area(D) = 2\pi (\cosh r - 1).
$$
Comparing this to $\area(S) = 4\pi (g-1)$ shows
$$
r < 2 \log(2g-1 + \sqrt{2g(2g-2)}) < 2 \log(4g-2).
$$
This weaker bound can be found in \cite{BuserBook}[Lemma 5.2.1] but note that the order of growth of this bound is the same as in Bavard's result.

Consider a hyperbolic surface $S$ possibly with geodesic boundary. In the free homotopy class of a simple closed geodesic loop $\gamma_x$ based at a point $x\in S$ lies a unique simple closed geodesic $\gamma$ (possibly a cusp or a boundary geodesic). In the event where $\gamma$ is not a cusp, it will be useful to bound the Hausdorff distance between $\gamma$ and $x$.

\begin{lemma}
Let $S,\gamma_x,\gamma$ be as above. Then
$$
\max_{y\in \gamma} d(x,y) < \arccosh \left( \cosh\frac{\ell(\gamma_x)}{2} \coth \frac{\ell(\gamma)}{2}\right).
$$
\end{lemma}

\begin{proof}
Note that $\gamma_x$ and $\gamma$ bound a cylinder that can be cut into two tri-rectangles with consecutive sides of length $\sfrac{\ell(\gamma)}{2}$ and $d(\gamma_x,\gamma)$ as in figure \ref{fig:loop}. 

\begin{figure}[h]
\leavevmode \SetLabels
\L(.755*.45) $d(\gamma_x,\gamma)$\\
\L(.51*.32) $\frac{\ell(\gamma)}{2}$\\
\L(.51*.84) $\frac{\ell(\gamma_x)}{2}$\\
\endSetLabels
\begin{center}
\AffixLabels{\centerline{\epsfig{file =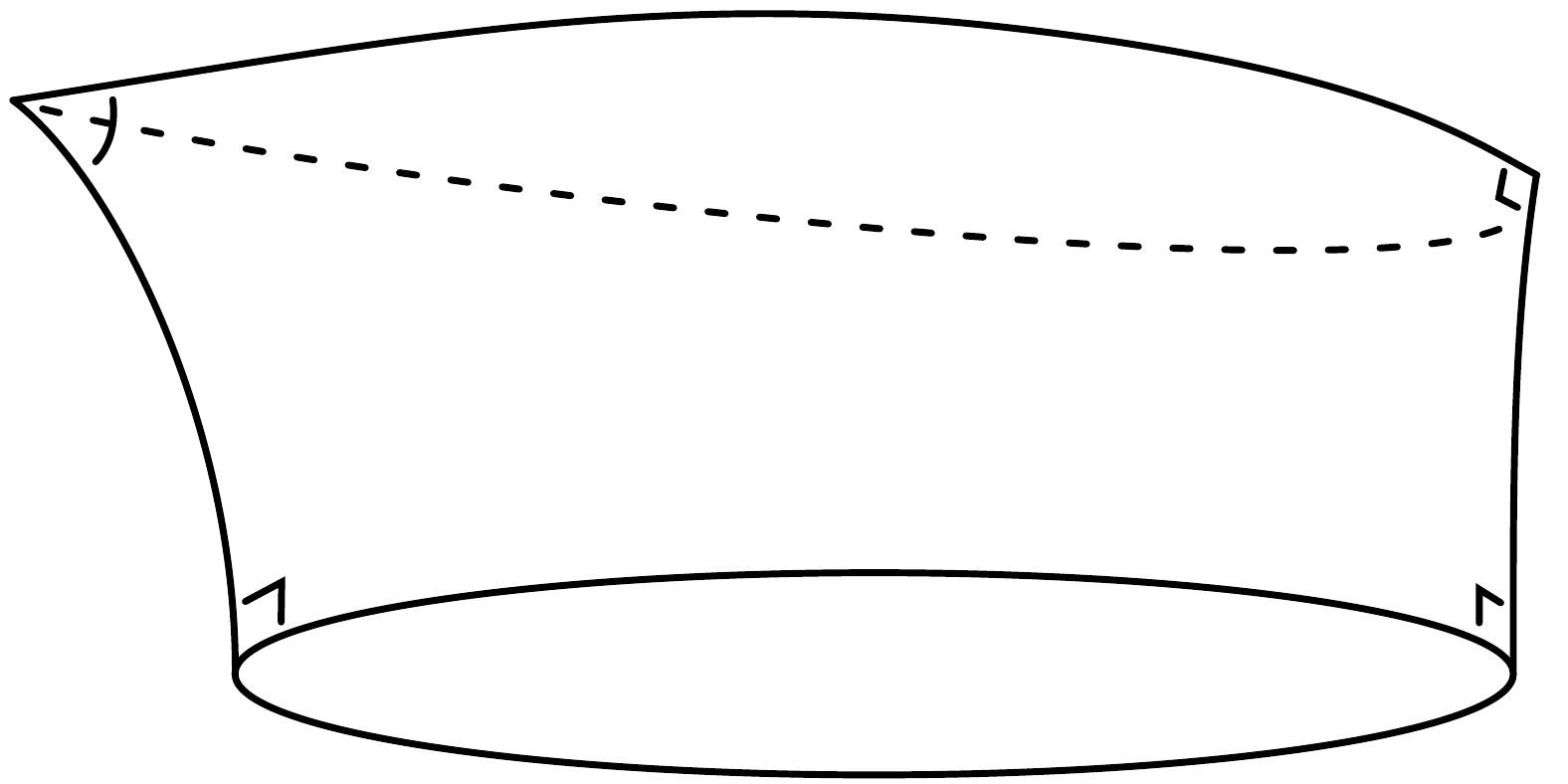,width=8.0cm,angle=0}}}
\vspace{-18pt}
\end{center}
\caption{From a geodesic loop to a closed geodesic} \label{fig:loop}
\end{figure}

Hyperbolic trigonometry in the tri-rectangle implies
\begin{equation}\label{eqn:tri}
\sinh d(\gamma_x,\gamma) \sinh \frac{\ell(\gamma)}{2} < 1.
\end{equation}
Now the maximum distance between $x$ and $\gamma$ is at most the length of the diagonal of the tri-rectangle. By hyperbolic trigonometry in one of the right angles triangles bounded by this diagonal, we obtain for all $y\in \gamma$:
$$
\cosh d(x,y) \leq \cosh d(\gamma_x,\gamma) \cosh \frac{\ell(\gamma_x)}{2}
$$
which via equation \ref{eqn:tri} becomes
\begin{eqnarray*}
d(x,y)& < &\arccosh \left(\cosh \left(\arcsinh \frac{1}{\sinh\frac{\ell(\gamma)}{2}}\right) \cosh \frac{\ell(\gamma_x)}{2}\right) \\
&=& \arccosh \left( \cosh\frac{\ell(\gamma_x)}{2} \coth \frac{\ell(\gamma)}{2}\right).
\end{eqnarray*}
\end{proof}

It is the following corollary of these lemmas that we shall use in the sequel. It is obtained by replacing $\ell(\gamma_x)$ by Bavard's bound, $\ell(\gamma)$ with $2 \arcsinh 1$ and by a simple manipulation.

\begin{corollary}\label{cor:distance}
Let $\gamma_x$ be the shortest geodesic loop based in $x\in S$ a closed surface and $\gamma$ the unique closed geodesic in its homotopy class. If $\ell(\gamma)\geq 2 \arcsinh 1$, then for all $y\in \gamma$
$$
d(x,y) < R_g:=\arccosh\frac{1}{\sqrt{2} \sin\frac{\pi}{12g-6}}.
$$
\end{corollary}

A further small manipulation gives the following rougher upper bound on this distance where the order of growth is more apparent:
$$
R_g < \log (4g-2) + \arcsinh 1.
$$

\section{Proof of main theorem}

We begin with any surface $S\in \M_g$ and our goal is to find a pants decomposition of $S$ which contains all simple closed geodesics of $S$ of length $\leq 2 \arcsinh 1$ and which has relatively short length. Recall that all simple closed geodesics of length less than $2 \arcsinh 1$ are disjoint, and it is for this reason that this value appears. Note that $S$ may have a pants decomposition of shorter length which doesn't contain all simple closed geodesics of length $\leq 2 \arcsinh 1$ but we restrict ourselves to searching for those that do. We'll call such pants decompositions {\it admissible} pants decompositions. 

As we are only looking among admissible pants decompositions, we can immediately apply Lemma \ref{lem:LE} to deform our surface $S$ to a new surface $S'$ with all simple closed geodesics of length greater or equal to $2 \arcsinh 1$ and with the length of all curves $\gamma$ lying in admissible pants decompositions of length at least $\ell_S(\gamma)$. (If $S$ already had this property, then $S'=S$.)

We now construct algorithmically a pants decomposition of $S'$. The algorithm has two main steps and a fail-safe step.

The algorithm is initiated as follows. Consider $x_1\in S'$ and $\gamma_{x_1}$ the shortest geodesic loop based at $x_1$. We set $\gamma_1$ to be the unique closed geodesic in the same free homotopy class and we cut $S'$ along $\gamma_1$ to obtain a surface with boundary (and possibly disconnected)
$$
S_1:= S'\setminus \gamma_1.
$$
Note that as such $S_1$ is an open surface but we could equivalently treat it as a compact surface with two simple closed geodesic boundary curves by considering its closure (but not its closure {\it inside} $S$). We will proceed in the sequel in a similar way.\\

\underline{Main Step 1}

Choose $x_{k+1}\in S_k$ with $d(x_{k+1},\partial S_k) > R_g$. Consider $\gamma_{x_{k+1}}$ the shortest geodesic loop in $x_{k+1}$. Observe that in light of Corollary \ref{cor:distance} $\gamma_{x_{k+1}}$ is not freely homotopic to any of the boundary curves of $S_k$. Set $\gamma_{k+1}$ to be the unique simple closed geodesic in the same free homotopy class and consider the surface
$$
S'_{k+1}:= S_k\setminus \gamma_{k+1}.
$$
We remove any pair of pants from $S'_{k+1}$ to obtain $S_{k+1}$

If there are no more remaining $x\in S_k$ with $d(x,\partial S_k) > R_g$ we proceed to the next main step, otherwise the step is repeated. For further reference we note that all curves created in this step have length at most 
$$2 \arccosh \frac{1}{2\sin
\frac{\pi}{12g-6}}$$
and thus in particular have length strictly less than $2R_g$.

\underline{Main Step 2}

All $x\in S_k$ satisfy $d(x,\partial S_k) \leq R_g$. Consider a point $x_{k+1} \in S_k$ such that there are two distinct geodesic paths that realize the distance from $x_{k+1}$ to $\partial S_k$. This provides a non-trivial simple path $c'$ from $\partial S_k$ to $\partial S_k$, where by non-trivial we mean that $S_k\setminus c'$ does not include a disk. In particular, in the free homotopy class of $c'$ with end points allowed to glide on the same boundary curves, there is a unique simple geodesic arc $c$ of minimal length, perpendicular in both end points to $\partial S_k$. 

There are two possible topological configurations for $c$ depending on whether $c$ is a path between two distinct boundary curves or not (see figure \ref{fig:toptypes} for an illustration). 

\begin{figure}[h]
\leavevmode \SetLabels
\L(.14*.79) $\alpha_1$\\
\L(.44*.79) $\alpha_2$\\
\L(.3*.08) $\tilde{\alpha}$\\
\L(.72*1.01) $\alpha$\\
\L(.535*.2) $\tilde{\alpha}_1$\\
\L(.845*.2) $\tilde{\alpha}_2$\\
\L(.3*.74) $c$\\
\L(.65*.3) $c$\\
\endSetLabels
\begin{center}
\AffixLabels{\centerline{\epsfig{file =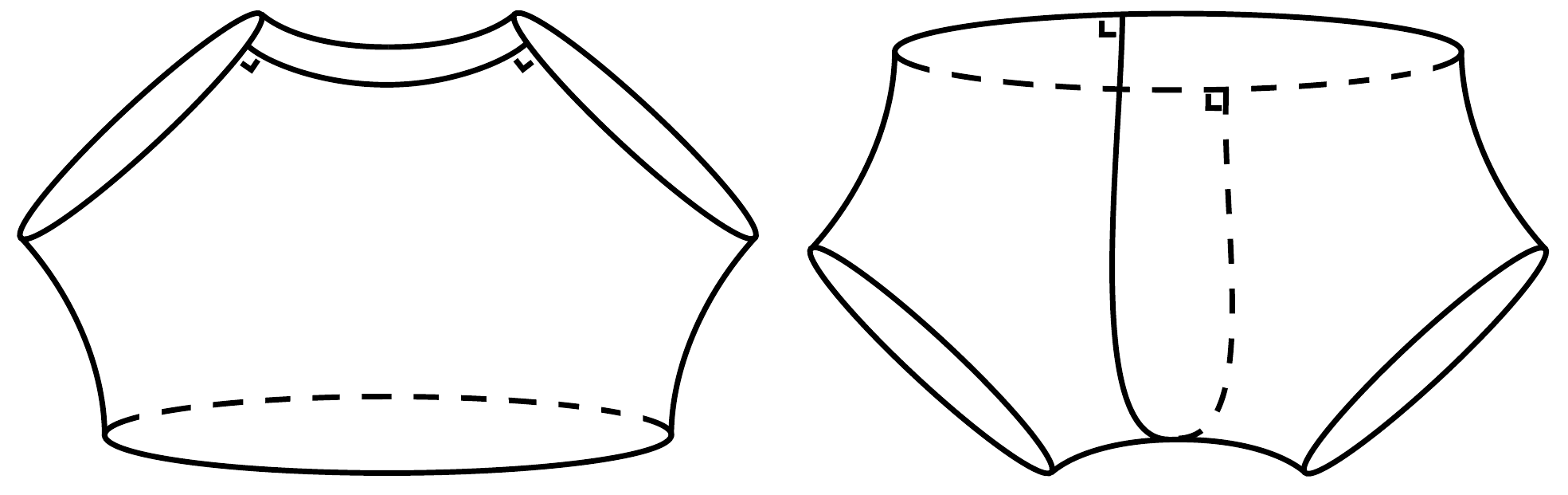,width=12.0cm,angle=0}}}
\vspace{-18pt}
\end{center}
\caption{The two topological types for path $c$} \label{fig:toptypes}
\end{figure}
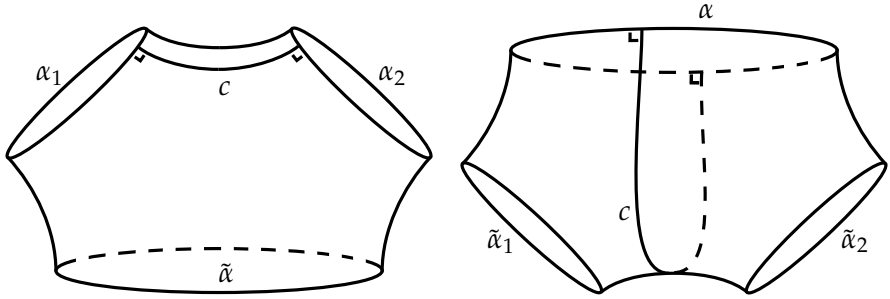

\underline{Case 1}: If $c$ is a path between distinct boundary curves $\alpha_1$ and $\alpha_2$, then $c\cup \alpha_1 \cup \alpha_2$ is contained in a unique pair of pants $(\alpha_1,\alpha_2,\tilde{\alpha})$. We set 
$$
\gamma_{k+1}:=\tilde{\alpha}
$$
and 
$$S'_{k+1}:=S_k\setminus (\alpha_1,\alpha_2,\tilde{\alpha}).
$$

\underline{Case 2}: If $c$ is a path with endpoints on a single boundary curve $\alpha$ then $c\cup \alpha$ is contained in a unique pair of pants $(\alpha,\tilde{\alpha}_1,\tilde{\alpha}_2)$. 

If $\tilde{\alpha}_1 \neq \tilde{\alpha}_2$ then we set
$$
\gamma_{k+1}:=\tilde{\alpha}_1, \gamma_{k+2}:=\tilde{\alpha}_2
$$
and 
$$S'_{k+2}:=S_k\setminus (\alpha_1,\alpha_2,\tilde{\alpha}).
$$

If $\tilde{\alpha}_1=\tilde{\alpha}_2$ then $(\alpha_1,\alpha_2,\tilde{\alpha})$ is contained in a one holed torus $T$ and we set 
$$
\gamma_{k+1}:=\tilde{\alpha}_1
$$
and 
$$S'_{k+1}:=S_k\setminus T.
$$

The algorithm continues until $\gamma_{3g-3}$ is constructed, i.e., when a full pants decomposition is obtained.\\

\underline{Lengths of curves}

Begin by observing that in both types of steps described above, at each step we have 
$$
\ell(\partial S_{k+1}) <  \ell(\partial S_{k})+4 R_g.
$$
Indeed: if $S_{k+1}$ is obtained by cutting along a curve as in Step 1, then the length of the curve is strictly shorter than $2R_g$ and the boundary increases by at most twice this length. 

If $S_{k+1}$ is obtained as in Step 2, case 1, then the curve $\tilde{\alpha}$ is of length at most 
$$\ell(\alpha_1)+\ell(\alpha_2)+4R_g.$$ 
As $S_{k+1}$ is obtained by removing the pair of pants with curves $\alpha_1, \alpha_2$ and $\tilde{\alpha}$, the boundary of $S_{k+1}$ no longer contains $\alpha_1$ and $\alpha_2$ and the boundary length increases by at most $4R_g$. In Step 2, case 2, one argues similarly.

In order to ensure that the length of the constructed curves does not surpass the desired length, the algorithm contains a fail-safe step. \\

\underline{Fail-safe step}

If at any step $\ell(\partial S_k) \geq 4\pi (g-1)$ then the next curve is constructed following slightly different procedure which we describe here. First observe that if 
$$\ell(\partial S_k) \geq 4\pi (g-1)$$ with $S_k$ obtained as above, then 
$$
\ell(\partial S_k) < 4\pi (g-1) +4 R_g
$$
as at every step boundary length cannot increase by more than $4 R_g$.

We consider an $r$-neighborhood of $\partial S_k$. For small enough $r$, this neighborhood is a union of cylinders around the boundary curves. We let $r$ grow until the topology changes, i.e., until the cylinder first bump into each other. We choose one of the geodesic paths $c$ created at the bumping point. 

Here we use an area argument to bound the length of $c$. The area of an embedded $r$-neighborhood of $\ell(\partial S_k)$ is at most that of the surface thus 
$$
\ell(\partial S_k) \sinh r < 4\pi(g-1).
$$
By assumption this implies
$$
r< \arcsinh 1
$$
and thus 
$$
\ell(c) < 2 \arcsinh 1.
$$

As before, there are two topological types for $c$, case 1 and case 2 as above. In both cases, we borrow the notation from above, but we argue slightly differently for the lengths. 

In case 1 we have a pair of pants with boundary curves $\alpha_1,\alpha_2$ and $\tilde{\alpha}$ which we decompose into two right angles hexagons. By the hexagon relations we have
\begin{eqnarray*}
\cosh \frac{\ell(\tilde{\alpha})}{2}&=& \sinh \frac{\ell(\alpha_1)}{2}\sinh \frac{\ell(\alpha_2)}{2} \cosh \ell(c) - \cosh \frac{\ell(\alpha_1)}{2}\cosh \frac{\ell(\alpha_2)}{2}\\
&<&  \sinh \frac{\ell(\alpha_1)}{2}\sinh \frac{\ell(\alpha_2)}{2} \, 3 - \cosh \frac{\ell(\alpha_1)}{2}\cosh \frac{\ell(\alpha_2)}{2}\\
&<& \cosh\left(  \frac{\ell(\alpha_1)}{2}+ \frac{\ell(\alpha_2)}{2}   \right).
\end{eqnarray*}
From this 
$$
\ell(\tilde{\alpha}) < \ell(\alpha_1)+ \ell(\alpha_2).
$$
So at this step we have 
$$
\ell(\partial S_{k+1}) <  \ell(\partial S_{k}).
$$
A similar (and easier) argument shows that the same conclusion holds in case 2 by looking at a pentagon decomposition of the pants $(\alpha,\tilde{\alpha}_1,\tilde{\alpha}_2)$.

Note that after a fail-safe step the boundary length decreases so it is possible that we return to Main Step 2 but otherwise we continue to create curves while decreasing the total boundary length. 

All the curves $\gamma_k$ created are at one point boundary curves of a surface $S'_k$ from Main Step 1, a boundary curve of a surface $S_k$ from Main Step 2 or a boundary curve of $S_k$ from the fail-safe step. As such their lengths are all bounded by the total boundary lengths of these surfaces. Thus
$$
\ell(\gamma_k) < 4\pi (g-1) +4 R_g
$$
and the theorem is proved.

\addcontentsline{toc}{section}{References}
\def\cprime{$'$}

\end{document}